\documentclass[11pt]{article}

\usepackage{times}
\usepackage{amsmath,amsfonts,amstext,latexsym,amssymb,amsbsy,amsopn,amsthm,eucal,slashed,amstext,enumerate,accents}
\usepackage{mathrsfs}
\usepackage[T1]{fontenc}
\usepackage{amssymb}
\usepackage{cases}
\usepackage{color}
\usepackage{txfonts}
\usepackage{dsfont}
\usepackage{graphicx}   \usepackage{comment}
\usepackage{amsthm}

\allowdisplaybreaks[4]
\usepackage[colorlinks,
linkcolor=blue,
anchorcolor=blue,
citecolor=blue,
]{hyperref}

\begin{document}

	
	\setlength{\abovedisplayshortskip}{3mm}
	\setlength{\belowdisplayshortskip}{3mm}
	\setlength{\abovedisplayskip}{3mm}
	\setlength{\belowdisplayskip}{3mm}
	
	\newtheorem{theorem}{Theorem}[section]
	\newtheorem{corollary}[theorem]{Corollary}
	\newtheorem{lemma}[theorem]{Lemma}
	\newtheorem{definition}[theorem]{Definition}
	\newtheorem{proposition}[theorem]{Proposition}
	\newtheorem{remark}[theorem]{Remark}
	\newtheorem{example}[theorem]{Example}
	\newtheorem{theoremalph}{Theorem}
	\renewcommand\thetheoremalph{\Alph{theoremalph}}
	\newtheorem{assumption}{Assumption}[section]

	\newenvironment{sequation}{\begin{equation}\small}{\end{equation}}
	\newenvironment{tequation}{\begin{equation}\tiny}{\end{equation}}
	\renewcommand{\thefootnote}{\fnsymbol{footnote}}
	

	\title
	{Gradient estimates and Liouville theorems for the \(\Phi\)-Laplacian equations on Riemannian manifolds }
	\author{
		Yu-Zhao Wang\footnote{Corresponding author: Research of Y.-Z. Wang has been supported by Fundamental Research Program of Shanxi Province(No. 202303021211001) .}
		\quad Jian-Hua Hao}
	
	\date{}

	\maketitle
	\maketitle \numberwithin{equation}{section}
	\maketitle \numberwithin{theorem}{section}
	\setcounter{tocdepth}{2}
	\setcounter{secnumdepth}{2}

	\begin{abstract}
		This paper establishes gradient estimates and Liouville-type theorems for the \(\Phi\)-Laplacian equation \(\Delta_{\Phi}(u) = G(|\nabla u|^2)\) on complete Riemannian manifolds and its parabolic counterpart \(\partial_t u = \Lambda_{\Phi}(u)\) on compact Riemannian manifolds. 
		Using a nonlinear \(\Phi\)-Bochner formula and the Nash-Moser iteration technique, we prove local gradient bounds under the lower bound assumption of Ricci curvature and suitable conditions on \(\Phi\) and \(G\), which leads to Liouville theorems for global solutions. 
		For the parabolic case, we employ the maximum principle to derive gradient estimates on compact Riemannian manifolds, and subsequently obtain Liouville-type results. Our work provides a unified framework that generalizes prior results for \(p\)-harmonic functions and other quasilinear equations.
		
		\vspace{3mm}

		\vspace{3mm}	
		
	\end{abstract}

	\section{\textbf{Introduction}}
	\quad Gradient estimates are fundamental and powerful tools in geometric analysis. In 1975,  Cheng and Yau \cite{Cheng1975,Yau1975} first proved the gradient estimate for positive harmonic functions on Riemannian
	manifolds. More precisely, if \(u\) is a positive solution to \(\Delta u=0\) on a complete Riemannian manifold \((M,g)\) with \({\rm Ric}_g\geq-(n-1)\kappa(\kappa\geq0)\), then on a geodesic ball $B(o, R)$ we have
	\begin{equation}\label{eq:Cheng-Yau}
		\frac{|\nabla u|}{u}\leq C_n\left(\frac{1+\sqrt{\kappa}R}{R}\right).
	\end{equation}
	In 1986,  Li and Yau \cite{Li1986} established  the following estimate (the Li-Yau estimate) for positive solutions to the heat equation \(\partial_tu=\Delta u\)  on an $n$-dimensional compact Riemannian manifold \(M\) with \({\rm Ric}_g\geq0\) 
	\begin{equation}\label{eq:Li-Yau}
		\frac{|\nabla u|^2}{u^2}-\frac{u_t}{u}\leq\frac{n}{2t}.
	\end{equation}
	Later,
	similar techniques were employed by R. Hamilton in the study of the Ricci flow, mean curvature flow and other geometric evolution equations; see the survey paper \cite{Ni2007}.  In particular,  Kotschwar and Ni \cite{Kotschwar2009} proved a  Cheng-Yau type estimate for \(p\)-harmonic equation  $\Delta_pu=0$ on complete Riemannian manifolds with
	the sectional curvature bounded from below via the classic maximum principle.
	Wang-Zhang\cite{Wang2011} extended this result to the case of Ricci curvature by using the Nash-Moser iteration technique. 
	
	In the same paper \cite{Kotschwar2009}, Kotschwar and Ni derived a Li-Yau type estimate for the \(p\)-Laplace heat equation 
	\begin{equation}\label{pheat}
		\frac{\partial v^{p-1}}{\partial t}=(p-1)^{p-1}{\rm div}\left(|\nabla v|^{p-2}\nabla v\right)
	\end{equation}
	on a compact Riemannian manifold with nonnegative Ricci curvature
	\begin{equation}\label{eq:ni estimate}
		(p-1)^{p}\frac{|\nabla v|^p}{v^p}-(p-1)\frac{v_t}{v}\leq\frac{n}{pt}.
	\end{equation}
	
	A natural question is whether gradient estimates can be obtained for other nonlinear
	equations. 
	In this research field, many scholars have obtained corresponding results, for instance, for the porous medium equation, fast diffusion equation\cite{LNVV2009}, doubly nonlinear diffusion equation \cite{WC2014} and so on. Inspired by the above works, we aim to study gradient estimates for more general nonlinear equations.

	In 1994, Caffarelli, et al \cite{Caffarelli1994} considered the following quasilinear
	equation in $\mathbb{R}^n$
	\begin{equation}\label{eq:quasi-linear 1.0}
		\Delta_{\Phi}(u)=F'(u)=f(u),
	\end{equation}
	where 
	\begin{equation}\label{PhiLa}
		\Delta_{\Phi}(u):={\rm div}\left(\Phi'(|\nabla u|^2)\nabla u\right)
	\end{equation}
	is called the $\Phi$-Laplacian. In particular, if $\Phi(x)=\frac{2}{p}x^{\frac{p}{2}}$, then $\Delta_{\Phi}(u)={\rm div}\left(|\nabla u|^{p-2}\nabla u\right)=\Delta_pu$. 
	The  solutions of equation \eqref{eq:quasi-linear 1.0} can be viewed as the
	critical points of the functional defined by
	\begin{equation}
		J(u):=\int_{\mathbb{R}^n}\left(\frac{1}{2}\Phi(|\nabla u|^2)+F(u)\right)dx,
	\end{equation}
	where the function \(\Phi\) in \eqref{eq:quasi-linear 1.0} is assumed to be normalized with \(\Phi(0)=0\). Using the
	maximum principle to some $P$-function, a Liouville-type theorem for the equation \eqref{eq:quasi-linear 1.0} was established \cite{Caffarelli1994}, along with the asymptotic behavior and uniqueness of its solutions.

	In the first part of this paper, we consider the following nonlinear elliptic equation on a complete Riemannian manifold
	\begin{equation}\label{eq:elliptic 1.1}
		\Delta_{\Phi}(u)=G(|\nabla u|^2).
	\end{equation}
	Throughout this paper, \(\Phi(x)\) and \(G(x)\) are assumed to be non-negative functions. When \(\Phi(x)=\frac{2}{p}x^{\frac{p}{2}}\) and \( G(x)=x^{\frac{q}{2}}\), then the equation \eqref{eq:elliptic 1.1} reduces to
	\begin{equation}\label{eq:elliptic variant}
		\Delta_pu-|\nabla u|^q=0,
	\end{equation}
	which is known as the quasi-linear Hamilton-Jacobi equation . When $p=q$, by setting $v=-(p-1)\log u$,  equation \eqref{eq:elliptic variant} reduce to the $p$-harmonic equation $\Delta_pv=0$.
	
	Using the Bernstein technique, Lions \cite{Lions1985} demonstrated that any \(C^2\) solution to \eqref{eq:elliptic variant} on \(\mathbb{R}^n\) with \(q>1\) and \(p=2\) must be constant. Bidaut-Veron, Garcia-Huidobro and Veron \cite{Bidaut-Veron2014} established the gradient estimate of solutions of \eqref{eq:elliptic variant} and obtained some Liouville-type theorems. Moreover, they extended their estimates to solutions of equation \eqref{eq:elliptic variant} on complete non-compact manifolds \((M^n,g)\)  under lower bounds depending on the Ricci curvature and sectional curvature etc.
	
	Recently, Y. Wang and his coauthors  \cite{Han2026} employed the Nash-Moser iteration to investigate gradient estimates for the quasi-linear equation
	\begin{equation}\label{eq:wang estimate}
		\Delta_pu+A|\nabla u|^q+B|u|^{r-1}u+C=0.
	\end{equation}
	It is shown that if the Ricci curvature is bounded from below and \(q>p-1>0\), then the equation \eqref{eq:wang estimate} satisfies the estimate
	\[
	\sup_{B_{R/2(o)}}|\nabla u|\leq C_{n,p,q,r}\left(\frac{1+\sqrt{\kappa}R}{R}\right)^{\frac{1}{q-p+1}}.
	\]
	Furthermore, in \cite{He2024,Shen2025,Wang2023c,Wang2025a,Wang2025b}, the authors obtained gradient estimates for various PDEs by employing the Nash-Moser iteration technique, while in \cite{Lin2024,Ma2026,WangZhang2025}, X. Ma and his coauthors established Liouville-type theorems for  \(p\)-Laplace type equations via the integral identity techniques. While in \cite{Sun2022}, Y. Sun and his coauthors studied second order quasilinear elliptic inequalities via a different approach and proved a sharp Liouville-type theorem.

	{A key distinction lies in the geometric assumptions: applying the Bernstein method necessitates the use of barrier functions and comparison theorems, which generally imposes constraints on the sectional curvature (see \cite{Bidaut-Veron2014, Ni2007}). In contrast, the Nash-Moser iteration method requires only requires the Ricci curvature. So we use the Nash-Moser iteration method to study the equation \eqref{eq:elliptic 1.1} on a complete Riemannian manifold with Ricci curvature bounded from below and obtained a local gradient estimate. A new feature of our derivation of the local gradient estimate is a nonlinear \(\Phi\)-Bochner type formula relating the nonlinear operator with its linearization.}
	
	\begin{theorem}\label{theorem:elliptic equation key}
		Let \((M^n,g)\) be a complete Riemannian manifold with \({\rm Ric}_g\geq-(n-1)\kappa g\) for some constant \(\kappa\geq0\). For any solution \(u\in B_R(o)\subset M\) to equation \eqref{eq:elliptic 1.1}, where \(B_R(o)\) denotes the geodesic ball at \(o\) with radius $R$.  Suppose that \(G>0\) on \(\mathbb{R}^{+}\), 
		and  \(\Phi,\Phi'\) satisfies Assumption \ref{a1}, define
		$$
		h(x):=\frac{\Phi''(x)}{\Phi'(x)},\quad and \quad \theta(x)=xh(x).
		$$
		If 
		\[
		\sup_{x\geq0}\left\{\left|\frac{(1+2xh(x))}{n-1}-x\frac{G'(x)}{G(x)}\right|\right\}\leq \delta,\quad G(x)-b_0\Phi'(x)x\geq0, 
		\]
		and 
		\[
		-\frac{1}{2}<b_1\leq \theta(x)\leq b_2<+\infty,
		\]
		then we have 
		\begin{equation}\label{Phiest}
			\sup_{B_{R/2}(o)}|\nabla u|\leq C_{n,\Phi,G}\left(\frac{1+\sqrt{\kappa}R}{R}\right),
		\end{equation}
		where \(\delta\), \(b_0\), \(b_1\), \(b_2\) are  constants and the constant \(C_{n,\Phi,G}\) depends on \(n,\Phi,G\).
	\end{theorem}
	
	\begin{remark}
		Theorem \ref{theorem:elliptic equation key}  extends several previous results.
		\begin{itemize}
			\item For \(\Phi(x)=\frac{2}{p}x^{\frac{p}{2}}\) and \(G(x)=x^{\frac{p}{2}}\), the estimate \eqref{Phiest} in Theorem \ref{theorem:elliptic equation key} reduce to the estimate in \cite{Wang2011} without additional assumptions on $\Phi$ and $G$.
			\item For \(\Phi(x)=\frac{2}{p}x^{\frac{p}{2}}\) and \(G(x)=x^{\frac{q}{2}}\), It suffices to estimate \(\mathcal{L}(w)\) and take \(\psi=w_{\varepsilon}^2\eta^2\) as the test function to obtain the result in \cite{Han2026} without additional assumptions on $\Phi$ and $G$. 
		\end{itemize}
	\end{remark}
	
	As a direct application, the following Liouville-type theorem is derived.
	\begin{theorem}\label{theorem:Liouville-type theorem}
		Under the same assumptions as in Theorem \ref{theorem:elliptic equation key} and \(\kappa=0\), 
		then exists  constant \(C=C(n,\Phi,G)\) such that 
		\begin{equation}\label{eq:1.3}
			\sup_{B_{R/2}(o)}|\nabla u|\leq C(n,\Phi,G)\left(\frac{1}{R}\right).
		\end{equation} 
		Moreover, if \(M\) is a complete manifold and \(u\) is a global solution of equation \eqref{eq:elliptic 1.1} on \(M\), then  letting \(R\rightarrow\infty\) in \eqref{eq:1.3} yields \(\nabla u\equiv0\) , then \(u\) is constant.	
	\end{theorem}
	
	\begin{theorem}\label{theorem:harnack inequality}
		Assume \((M^n,g)\) satisfies the same assumption as in Theorem \ref{theorem:elliptic equation key}. Let \(u\) be a global solution to the equation \eqref{eq:elliptic 1.1} on \(M\). Then, for fixed \(o\in M\) and any \(x\in M\) we have 
		\begin{equation}\label{eq:1.4}
			u(o)-c(n,\Phi,G)\kappa^{\frac{1}{2}}d(x,o)\leq u(x)\leq u(o)+c(n,\Phi,G)\kappa^{\frac{1}{2}}d(x,o).	
		\end{equation}
	\end{theorem}

	In the second part of this paper, taking \(G(|\nabla u|^2)=\frac{p}{2}\Phi(|\nabla u|^2)\) in equation \eqref{eq:elliptic 1.1}, we obtain the following nonlinear parabolic equation
	\begin{equation}\label{eq:parabolic 1.2}
		\frac{\partial u}{\partial t}-\Lambda_{\Phi}(u)=0,
	\end{equation}
	where
	$$
	\Lambda_{\Phi}(u):={\rm div}\left(\Phi'(|\nabla u|^2)\nabla u\right)-\frac{p}{2}\Phi(|\nabla u|^2).
	$$	
	When \(\Phi(x)=\frac{2}{p}x^{\frac{p}{2}}\),  equation \eqref{eq:parabolic 1.2} reduces to
	\begin{equation}\label{eq:parabolic variant}
		\frac{\partial u}{\partial t}=\Delta_pu-|\nabla u|^p,
	\end{equation}
	which is equivalent to the $p$-heat equation \eqref{pheat}
	by setting  \(v=-(p-1)\log u\).

	Inspired by the above works, we apply the maximum
	principle to derive Li-Yau type gradient estimates for equation \eqref{eq:parabolic 1.2} on compact Riemannian manifolds respectively.

	\begin{theorem}\label{corollary:parabolic 1.1} 
		Let \((M^n,g)\) be a compact Riemannian manifold with \({\rm Ric}_g\geq0\). For any solution \(u\in M\times\mathbb{R}^{+}\) to the equation \eqref{eq:parabolic 1.2} and any $p>1$. Suppose that \(\Phi,\Phi'\) satisfies Assumption \ref{a1}, define
		$$
		h(x):=\frac{\Phi''(x)}{\Phi'(x)},\quad and \quad \theta(x)=xh(x).
		$$
		If 
		\[
		-\frac{1}{2}<b_1\leq \theta(x)\leq b_2<+\infty,
		\]
		we have	
		\begin{equation}\label{pe}
			\frac{p}{2}\Phi(|\nabla u|^2)+ u_t\leq\frac{n}{pt}.
		\end{equation}
	\end{theorem}
	
	\begin{remark}
		For \(\Phi(x)=\frac{2}{p}x^{\frac{p}{2}}\), the estimate \eqref{pe} in Theorem \ref{corollary:parabolic 1.1} reduce to the estimate in \cite{Kotschwar2009} without additional assumptions on $\Phi$.
	\end{remark}

	\begin{corollary}\label{corollary:parabolic 1.2}
		Under the same notations and assumptions as in Theorem \ref{corollary:parabolic 1.1} , if \(\Phi(x)\geq\frac{2}{p}x\), then for all \(x_1,x_2\in M\) and \(0<t_1<t_2<+\infty\),
		\[
		u(x_2,t_2)-u(x_1,t_1)\leq\frac{n}{p}\ln\frac{t_2}{t_1}+\frac{d(x_1,x_2)^2}{t_2-t_1}.
		\]
	\end{corollary}
	
	\begin{remark}
		When \(\Phi(x)=\frac{2}{p}x^{\frac{p}{2}}\),  using Young's  inequality, we have
		\[
		u(x_2,t_2)-u(x_1,t_1)\leq\frac{n}{p}\ln\frac{t_2}{t_1}+(p-1)p^{-\frac{p}{p-1}}\left(\frac{d(x_1,x_2)}{t_2-t_1}\right)^{\frac{p}{p-1}}.
		\]
	\end{remark}

	Moreover, we deuce the following Liouville theorem from Theorem \ref{corollary:parabolic 1.1}.
	\begin{theorem}\label{theorem: parabolic Liouville-type theorem}
		Under the same assumptions as in Theorem \ref{corollary:parabolic 1.1} . Assume that \(u\) is a positive smooth solution to time-independent equation \(\Delta_{\Phi}u-\frac{p}{2}\Phi(|\nabla u|^2)=0\) on \(M\). If there exist two positive constants \(c\) and \(C\) such that \(c<u<C\), then \(u\) must be constant.	
	\end{theorem}
	
	During the finalization of our manuscript, a related study \cite{Shen2025} appeared, which investigates gradient estimates for the similar equation \({\rm{div}}\left(\varphi(|\nabla u|^2)\nabla u\right)+\psi(u^2)u=0\). While both works share the Nash-Moser framework, our approach differs in key technical
	details and yields a more natural and concise result.

	The paper is organized as follows. In Section \ref{S2}, we develop the variational formulation for the \(\Phi\)-Laplacian operator \eqref{PhiLa} and derive its linearized operator together with a \(\Phi\)-Bochner type formula. In Section \ref{S3},  we establish gradient estimates, Liouville-type theorems, and Harnack inequalities via the Nash-Moser iteration for solutions to elliptic \(\Phi\)-Laplace equation \eqref{eq:elliptic 1.1}. In Section \ref{S4}, we extend the analysis to the parabolic setting, applying maximum principle to obtain gradient estimates, Liouville-type theorems, and Harnack inequalities for parabolic \(\Phi\)-Laplace equations \eqref{eq:parabolic 1.2}.
	
	\section{Variational Formulation and Bochner Formula for the \(\Phi\)-Laplacian}\label{S2}
	
	Throughout this paper, let \((M^n,g)\) be an \(n\)-dimensional Riemannian manifold with \({\rm Ric}_g\geq-(n-1)\kappa g\) for some constant \(\kappa\geq0\).  The volume form  is given by \(dV=\sqrt{\det(g_{ij})}dx_1\wedge...\wedge dx_n\), where \((x_1,...,x_n)\) is a local coordinate chart.  For simplicity, we may omit the volume form when integrating over \(M\).

	For a smooth function $\psi$ on a Riemannian manifold $(M^n,g)$,
	\begin{align*}
		\frac{d}{d\varepsilon}\Big{|}_{\varepsilon=0}\Delta_{\Phi}(u+ \varepsilon\psi) 
		&= \frac{d}{d\varepsilon}\Big{|}_{\varepsilon=0}{\rm div}\Big{(}\Phi ^{\prime}(|\nabla(u+\varepsilon\psi)|^{2})\nabla(u+\varepsilon\psi)\Big{)} \\
		&= {\rm div}\Big{(}\Phi^{\prime}(|\nabla u|^{2})\nabla\psi+2\Phi^{\prime\prime}(|\nabla u|^{2})(\nabla u\cdot\nabla\psi)\nabla u\Big{)} \\
		&= {\rm div}\Big{(}\Phi^{\prime}(w)A(\nabla\psi)\Big{)},
	\end{align*}
	where $w=|\nabla u|^{2}$,  $h(w)=\frac{\Phi^{\prime\prime}(w)}{\Phi^{\prime}(w)}$ and $A$ is a 2-tensor defined by
	\begin{equation}\label{eq:2.1}
		A:=\mathrm{Id}+2\frac{\Phi^{\prime\prime}(w)}{\Phi^{\prime}(w)}\nabla u\otimes \nabla u=\mathrm{Id}+2h(w)\nabla u\otimes \nabla u,
	\end{equation}
	where \(\Phi'(w)>0\). Thus, the linearized operator of $\Delta_{\Phi}$ at $u$ is defined by
	\begin{equation}\label{eq:2.2}
		\begin{aligned}
		\mathcal{L}(\psi) &:= \mathrm{div}\Big(\Phi'(w)A(\nabla\psi)\Big)\\
				&= \Phi'(w)\Delta\psi + \Phi''(w)\langle \nabla w, \nabla \psi \rangle + 2\Phi'''(w)\langle \nabla u, \nabla \psi \rangle \langle \nabla w, \nabla u \rangle \\
				&\quad + 2\Phi''(w)\Big(\langle \nabla^2 u \nabla\psi, \nabla u \rangle
				+ \langle \nabla^2\psi \nabla u, \nabla u \rangle + \langle \nabla u, \nabla \psi \rangle \Delta u\Big).
		\end{aligned}
	\end{equation}

	The following $\Phi$-Bochner formulae for the linearized operator $\mathcal{L}$ play an important role in the proof.
	
	\begin{proposition}\label{proposition:Phi-Bochner}
		Set \(w=|\nabla u|^{2}\), we have
		\begin{equation}\label{eq:2.5}
			\mathcal{L}w=2\Phi'(w)\left(|\nabla\nabla u|^2+{\rm Ric}(\nabla u,\nabla u)\right)+2\langle\nabla\Delta_{\Phi}u,\nabla u\rangle+\Phi''(w)|\nabla w|^2,
		\end{equation}
		and
		\begin{equation}\label{eq:2.6}
			\mathcal{L}(\Phi(w))=2(\Phi'(w))^2(|\nabla\nabla u|_A^2+{\rm Ric}(\nabla u,\nabla u))+2\Phi'(w)\langle\nabla\Delta_{\Phi}u,\nabla u\rangle,
		\end{equation}
		where
		\begin{equation}\label{eq:2.7}
			|\nabla\nabla u|_A^2=A^{ik}A^{jl}u_{ij}u_{kl}=|\nabla\nabla u|^2+h(w)|\nabla w|^2+h^2(w)\langle\nabla w,\nabla u\rangle^2. 
		\end{equation}
	\end{proposition}
	\begin{proof}
		By a direct calculation and the Bochner formula, we have
		\begin{equation}\label{eq:2.8}
			\begin{aligned}
				\mathcal{L}w &= \mathrm{div}\Big{(}\Phi^{\prime}(w)\nabla w+2\Phi^{\prime\prime}(w)(\nabla w,\nabla u)\nabla u\Big{)} \\
				&= \Phi^{\prime}(w)\Delta w+\Phi^{\prime\prime}(w)|\nabla w|^{2}+2\Phi^{\prime\prime\prime}(w)\langle\nabla u,\nabla w\rangle^{2} \\
				&\quad+2\Phi^{\prime\prime}(w)\big{(}\langle\nabla w,\nabla u\rangle\Delta u+\langle\nabla(\nabla u,\nabla w),\nabla u\rangle\big{)}\\
				&= 2\Phi^{\prime}(w)\big{(}|\nabla\nabla u|^{2}+\mathrm{Ric}(\nabla u,\nabla u)+\langle\nabla\Delta u,\nabla u\rangle\big{)}+\Phi^{\prime\prime}(w)|\nabla w|^{2} \\
				&\quad +2\Phi^{\prime\prime}(w)\big{(}\langle\nabla w,\nabla u\rangle\Delta u+\langle\nabla(\nabla u,\nabla w),\nabla u\rangle\big{)}+2\Phi^{\prime\prime\prime}(w)\langle\nabla u,\nabla w\rangle^{2},
			\end{aligned}
		\end{equation}			
		On the other hand, by the definition of $\Delta_{\Phi}$,
		\begin{equation}\label{eq:2.9}
			\begin{aligned}
				2\langle\nabla\Delta_{\Phi}u,\nabla u\rangle &= 2\left\langle\nabla\Big{(}\Phi^{\prime}(w)\Delta u+\Phi^{\prime\prime}(w)(\nabla u,\nabla w)\Big{)},\nabla u\right\rangle \\
				&= 2\Phi^{\prime}(w)\langle\nabla\Delta u,\nabla u\rangle+2\Phi^{\prime\prime}(w)\langle\nabla w,\nabla u\rangle\Delta u \\
				&\quad +2\Phi^{\prime\prime\prime}(w)\langle\nabla w,\nabla u\rangle^{2}+2\Phi^{\prime\prime}(w)\langle\nabla(\nabla u,\nabla w),\nabla u\rangle.
			\end{aligned}
		\end{equation}
		Substituting \eqref{eq:2.9} into \eqref{eq:2.8} yields the desired $\Phi$-Bochner formula \eqref{eq:2.5}.
		
		For a function $U(w)$ of $w$, applying $\Phi$-Bochner formula \eqref{eq:2.5}, we have
		\begin{equation}\label{eq:2.10}
			\begin{aligned}
				\mathcal{L}U(w) &= \mathrm{div}\Big{(}\Phi^{\prime}(w)A(\nabla U(w))\Big{)} \\
				&=U^{\prime}(w)\mathcal{L}w+U^{\prime\prime}(w)\Phi^{\prime}(w)A(\nabla w)\cdot\nabla w \\
				&= U^{\prime}(w)\Big{(}2\Phi^{\prime}(w)(|\nabla\nabla u|^{2}+\mathrm{Ric}(\nabla u,\nabla u))+2\langle\nabla\Delta_{\Phi}u,\nabla u\rangle+\Phi^{\prime\prime}(w)|\nabla w|^{2}\Big{)}\\
				&\quad+U^{\prime\prime}(w)\Phi^{\prime}(w)|\nabla w|_{A}^{2}\\ 
				&= 2U^{\prime}(w)\Phi^{\prime}(w)\big{(}|\nabla\nabla u|^{2}+\mathrm{Ric}(\nabla u,\nabla u)\big{)}+2U^{\prime}(w)\langle\nabla\Delta_{\Phi}u,\nabla u\rangle \\
				&\quad +(U^{\prime}(w)\Phi^{\prime}(w))^{\prime}|\nabla w|^{2}+2U^{\prime\prime}(w)\Phi^{\prime\prime}(w)\langle\nabla u,\nabla w\rangle^{2},
			\end{aligned}
		\end{equation}
		where
		\[
		|\nabla w|_{A}^{2}:=A^{ij}w_{i}w_{j}=|\nabla w|^{2}+2h(w)\langle\nabla u,\nabla w\rangle^{2}.
		\]
		In particular, taking $U(w)=\Phi(w)$ in \eqref{eq:2.10}, we obtain
		\begin{align*}
			\mathcal{L}\Phi(w) &= 2\langle\Phi^{\prime}(w)\rangle^{2}\Big{(}|\nabla\nabla u|^{2}+\mathrm{Ric}(\nabla u,\nabla u)\Big{)}+2\Phi^{\prime}(w)\langle\nabla\Delta_{\Phi}u,\nabla u\rangle \\
			&\quad +2\Phi^{\prime}(w)\Phi^{\prime\prime}(w)|\nabla w|^{2}+2\langle\Phi^{\prime\prime}(w)\rangle^{2}\langle\nabla u,\nabla w\rangle^{2} \\
			&= 2\langle\Phi^{\prime}(w)\rangle^{2}\Big{(}|\nabla\nabla u|_{A}^{2}+\mathrm{Ric}(\nabla u,\nabla u)\Big{)}+2\Phi^{\prime}(w)\langle\nabla\Delta_{\Phi}u,\nabla u\rangle,
		\end{align*}
		where $|\nabla\nabla u|_{A}^{2}$ is defined in \eqref{eq:2.7}.	
	\end{proof}
	\begin{corollary}[\(p\)-Bochner formula]\label{corollary:p-Bochner-formula}
		When \(\Phi(x)=\frac{2}{p}x^{\frac{p}{2}}\), \(\Phi\)-Bochner formula in \eqref{eq:2.6} reduces to  the classic \(p\)-Bochner formula \cite{Wang2023b}
		\[
		\mathcal{L}(|\nabla u|^p)=p|\nabla u|^{2p-4}\left(|\nabla\nabla u|_A^2+{\rm Ric}(\nabla u,\nabla u)\right)+p|\nabla u|^{p-2}\langle\nabla\Delta_pu,\nabla u\rangle.
		\] 
	\end{corollary}

    \begin{assumption}\label{a1}
    	There exist \(p>1\), \(a\ge 0\), and constants \(c_1,c_2>0\) such that \(\Phi\in C^3(R^+)\) and for every \(\sigma,\xi\in R^n\backslash\{0\}\)
    	\begin{itemize}
    		\item \(c_1\left(a+|\sigma|\right)^{p-2}\le\Phi'(|\sigma|^2)\le c_2\left(a+|\sigma|\right)^{p-2};\)
    		\item \(c_1\left(a+|\sigma|\right)^{p-2}|\xi|^2\le\sum_{i,j=1}^na_{ij}(\sigma)\xi_i\xi_j\le c_2\left(a+|\sigma|\right)^{p-2}|\xi|^2,\)
    	\end{itemize}
    	where 
    	\[
    	a_{ij}(\sigma)=2\Phi''(|\sigma|^2)\sigma_i\sigma_j+\Phi'(|\sigma|^2)\delta_{ij}.
    	\]
    \end{assumption}
    
	\begin{definition}
		For a \(C^3\)-function \(\Phi\) on \([0,+\infty)\), the function of $\theta$ is defined by
		\[
		\theta(x):=xh(x).
		\]
		We say \(\Phi\) has finite lower bound and upper bound if there exist finite constants such that
		\[
		\inf_{x\geq0}\theta(x)=b_1\geq-\frac{1}{2},\quad \sup_{x\geq0}\theta(x)=b_2<+\infty.
		\]
	\end{definition}
	
	In our proof of gradient estimates for solutions to \eqref{eq:elliptic 1.1}, the following Sobolev inequality due to Saloff-Coste plays an important role.
	
	\begin{lemma}[Saloff-Coste\cite{Saloff-Coste1992}]\label{lemma:Saloff-Coste}
		Let \((M^n,g)\) be a complete manifold with \({\rm Ric}_g\geq-(n-1)\kappa g\). For \(n>2\), there exists a constant \(c_0\) depending only on \(n\), such that for any geodesic ball \(B\subset M\) of radius $R$ and volume $V$ and for any \(u\in C_0^{\infty}(B)\),  
		\begin{equation}\label{SCI}
			\|u\|_{L^{\frac{2n}{n-2}}{(B)}}^2\leq e^{c_0(1+\sqrt{\kappa}R)}V^{-\frac{2}{n}}R^2\left(\int_{B}|\nabla u|^2+R^{-2}u^2dV\right).
		\end{equation}
		For \(n=2\), the above inequality holds with \(n\) replaced by any fixed \(n'>2\).
	\end{lemma}
	
	\begin{remark}
		For any open region \(\Omega\subset M\), if there exists a geodesic ball \(B\) such that \(\Omega\subset B\), then for any \(u\in W^{1,2}(\Omega)\),  
		\[
		\|u\|_{L^{\frac{2n}{n-2}}{(\Omega)}}^2\leq e^{c_0(1+\sqrt{\kappa}R)}V^{-\frac{2}{n}}R^2\left(\int_{\Omega}|\nabla u|^2+R^{-2}u^2dV\right).
		\]
		This can be seen by  choosing \(\{u_n\}\subset C_0^{\infty}(\Omega)\subset C_0^{\infty}(B)\) such that \(u_n\rightarrow u\) in \(W^{1,2}(\Omega)\).
	\end{remark}

	\section{Gradient estimates for the elliptic equation}\label{S3}
	
	In this section, we divide the proof of Theorem \ref{theorem:elliptic equation key} into three parts. First, we derive a basic integral inequality for \(w=|\nabla u|^2\), which will be used in the second and third parts. Second, we give an \(L^{\alpha_1}\)-estimate of \(w\) on a geodesic ball with radius \(3R/4\), where the \(L^{\alpha_1}\) norm of \(w\) determines the initial state of the Nash-Moser iteration. Finally, we complete the proof of our theorem by using the Nash-Moser iteration method.
	
	From now on, we always assume that \(\Omega=B_R(o)\subset M\) is a geodesic ball, and we  use \(a_i(i=1,2,3...)\) to denote some positive constants depending only \(n,\Phi,G\) and \(R\). 
	\begin{definition}
		A function \(u \in W^{1,1}_{\mathrm{loc}}(M)\) is called a \emph{weak solution} 
		of equation \eqref{eq:elliptic 1.1} if
		\(\Phi'(|\nabla u|^2)|\nabla u| \in L^{1}_{\mathrm{loc}}(M)\) and
		\[
		-\int_{M}\Phi'(|\nabla u|^2)\langle\nabla u,\nabla\psi\rangle\,dV
		-\int_{M}G(|\nabla u|^2)\psi\,dV = 0
		\qquad \forall\,\psi\in C_0^{\infty}(M).
		\]
	\end{definition}
	It is worth mentioning that, under suitable ellipticity and growth conditions
	on \(\Phi\) and \(G\) (see \cite{Cianchi2018}), any weak solution \(u\) 
	satisfies \(u\in W_{\mathrm{loc}}^{2,2}(M)\) and \(u\in C^{1,\alpha}(M)\) 
	for some \(\alpha\in(0,1)\). Moreover, away from the set \(\{\nabla u=0\}\),
	\(u\) is in fact smooth. 
	\begin{lemma}\label{lemma:G-Phi scaling}
		Let $\Phi \in C^3(R^+)$, $G \in C^2(R^+)$ with $\Phi'(x) > 0$ and $G(x) > 0$ for all $x \in [0,\infty)$. Suppose $b_0 > 0$ is a constant. If the differential inequality
		\[
		G(x)-b_0 \Phi'(x) x \ge 0 
		\]
		holds for all $x \in [0,\infty)$, then
		\[
		\left(\frac{G(x)}{\Phi'(x)}\right)^2 \ge b_0^2 x^2.
		\]
	\end{lemma}
	\begin{proof}
		From the condition $G(x) \ge b_0 \Phi'(x) x$ and $\Phi'(x) > 0$, dividing both sides by $\Phi'(x)$ yields $\frac{G(x)}{\Phi'(x)} \ge b_0 x$. Since $b_0 \ge 0$ and $x \ge 0$, the right-hand side is non-negative, and the left-hand side is also non-negative because $G(x) \ge 0$ and $\Phi'(x) > 0$. Squaring both sides gives the desired inequality.
	\end{proof}

	\begin{lemma}\label{lemma:2.1}
		If \(u\) is a solution to equation \eqref{eq:elliptic 1.1} and \(w=|\nabla u|^{2}\), then 
		\begin{equation}\label{eq:2.11}
			\mathcal{L}w=2\Phi'(w)\left(|\nabla\nabla u|^2+{\rm Ric}(\nabla u,\nabla u)\right)+2G'(w)\langle\nabla w,\nabla u\rangle+\Phi''(w)|\nabla w|^2,
		\end{equation}
		and
		\begin{equation}\label{eq:2.12}
			\mathcal{L}(\Phi(w))=2(\Phi'(w))^2(|\nabla\nabla u|_A^2+{\rm Ric}(\nabla u,\nabla u))+2\Phi'(w)G'(w)\langle\nabla w,\nabla u\rangle.
		\end{equation}
	\end{lemma}
	\begin{proof} 
		Equations \eqref{eq:2.11} and \eqref{eq:2.12} follow directly from Proposition \ref{proposition:Phi-Bochner}.
	\end{proof}
	Furthermore, we need the following pointwise estimate for \(\mathcal{L}(\Phi(w))\).
	
	\begin{lemma}\label{L-Phi-estimate}
		If \(u\) is a solution to equation \eqref{eq:elliptic 1.1} on a Riemannian manifold with  \({\rm Ric}_g\geq-(n-1)\kappa g\) for some constant \(\kappa\geq0\), then on the set \(\{w\neq0\}\) , we have		
		\begin{equation}\label{eq:2.14}
			\mathcal{L}(\Phi(w))\geq\frac{2}{n-1}G^2(w)-2(n-1)\kappa(\Phi'(w))^2w-2\delta w^{-\frac{1}{2}}\left|\Phi'(w)\right|\left|G(w)\right||\nabla w|.
		\end{equation}
	\end{lemma}
	
	\begin{proof}
		Let \(\{e_1,e_2,...,e_n\}\) be a local orthonormal frame of \(TM\) in a domain with \(w\neq0\) such that \(e_1=\frac{\nabla u}{|\nabla u|}\). Then  \(u_1=w^{\frac{1}{2}}\) and
		\begin{equation}\label{eq:2.15}
			u_{11}=\frac{1}{2}w^{-\frac{1}{2}}w_1=\frac{1}{2}w^{-1}\langle\nabla u,\nabla w\rangle.
		\end{equation}	
		The \(\Phi\)-Laplace operator can be expressed in terms of \(w\),
		\begin{equation}\label{eq:2.16}
			\begin{aligned}
				\Delta_{\Phi}u&={\rm div}(\Phi'(w)\nabla u)=\Phi''(w)\langle\nabla u,\nabla w\rangle+\Phi'(w)\Delta u\\
				&=(2w\Phi''(w)+\Phi'(w))u_{11}+\Phi'(w)\sum_{i=2}^{n}u_{ii}.
			\end{aligned}
		\end{equation}	
		Substituting \eqref{eq:2.16} into equation \eqref{eq:elliptic 1.1} yields
		\begin{equation}\label{eq:2.17}
			\sum_{i=2}^{n}u_{ii}=\frac{G(w)}{\Phi'(w)}-(2wh(w)+1)u_{11},
		\end{equation}	
		where $h(w)=\frac{\Phi^{\prime\prime}(w)}{\Phi^{\prime}(w)}$. Using  \(u_1=w^{\frac{1}{2}}\) again, we have 
		\begin{equation}\label{eq:2.18}
			|\nabla w|^2=\sum_{i=1}^{n}|2u_1u_{1i}|^2=4w\sum_{i=1}^{n}u_{1i}^2.
		\end{equation}	
		By \eqref{eq:2.15} and \eqref{eq:2.18},  we have
		\begin{align*}
			|\nabla\nabla u|_A^2&=|\nabla\nabla u|^2+h(w)|\nabla w|^2+h^2(w)\langle\nabla w,\nabla u\rangle\notag\\
			&=\sum_{i,j=1}^{n}u_{ij}^2+4wh(w)\sum_{i=1}^{n}u_{1i}^2+4w^2h^2(w)u_{11}^2\notag\\
			&=\sum_{i,j=2}^{n}u_{ij}^2+2(1+2wh(w))\sum_{i=2}^{n}u_{1i}^2+(1+2wh(w))^2u_{11}^2.
		\end{align*}
		
		By combining \eqref{eq:2.17} with the Cauchy-Schwarz inequality, we get
		\begin{align*}\label{eq:2.19}
			|\nabla\nabla u|_A^2	\geq&\frac{1}{n-1}\left(\sum_{i=2}^{n}u_{ii}\right)^2+2(1+2wh(w))\sum_{i=2}^{n}u_{1i}^2+(1+2wh(w))^2u_{11}^2\notag\\
			=&\frac{1}{n-1}\left(\frac{G(w)}{\Phi'(w)}-(1+2wh(w))u_{11}\right)^2+2(1+2wh(w))\sum_{i=2}^{n}u_{1i}^2+(1+2wh(w))^2u_{11}^2\notag\\
			=&\frac{1}{n-1}\left(\frac{G(w)}{\Phi'(w)}\right)^2-\frac{2(1+2wh(w))}{n-1}\frac{G(w)}{\Phi'(w)}u_{11}+2(1+2wh(w))\sum_{i=2}^{n}u_{1i}^2+\frac{n}{n-1}(1+2wh(w))^2u_{11}^2\notag\\
			\geq&\frac{1}{n-1}\left(\frac{G(w)}{\Phi'(w)}\right)^2-\frac{(1+2wh(w))}{n-1}\frac{G(w)}{\Phi'(w)}w^{-1}\langle\nabla u,\nabla w\rangle+\frac{a_0}{4w}|\nabla w|^2,
		\end{align*}
		where \(a_0=\min\left\{2(1+2wh(w)),\frac{n}{n-1}(1+2wh(w))^2\right\}=\min\left\{2(1+2b_1),\frac{n(1+2b_1)^2}{n-1}\right\}>0\), thus
		\begin{equation}\label{eq:2.20}
			\begin{aligned}
				|\nabla\nabla u|_A^2\geq\frac{1}{n-1}\left(\frac{G(w)}{\Phi'(w)}\right)^2-\frac{(1+2wh(w))}{n-1}\frac{G(w)}{\Phi'(w)}w^{-1}\langle\nabla u,\nabla w\rangle+\frac{a_0}{4w}|\nabla w|^2.	
			\end{aligned}
		\end{equation}	
		Using the Ricci curvature assumption, \eqref{eq:2.14} follows from \eqref{eq:2.12} and \eqref{eq:2.20}.
	\end{proof}

	\begin{lemma}\label{lemma:integral inequality}
		Let \(u\) be a solution to equation \eqref{eq:elliptic 1.1} on \(\Omega\subset M\). Then, for any positive number \(\alpha\), there exist constants \(a_1=\min\left\{1 , (1+2b_1)\left(1-\frac{2b_2}{\alpha}\right)\right\}\), \(a_2=1+2b_2\), and \(a_3\) such that  
		$\frac{4\delta^2}{a_1\alpha}\leq\frac{2}{n-1}$ and for any non-negative cut-off function \(\eta\in C_0^{\infty}(\Omega)\), 
		\begin{align}\label{eq:3.1}
			&\frac{a_1\alpha}{(\alpha+1)^2}\int_{\Omega}\left|\nabla\left(w^{\frac{\alpha}{2}+\frac{1}{2}}\eta\right)\right|^2dV+\frac{1}{n-1}\int_{\Omega}\left(\frac{G(w)}{\Phi'(w)}\right)^2w^{\alpha}\eta^2dV\notag\\
			\leq&2(n-1)\kappa\int_{\Omega}w^{\alpha+1}\eta^2dV+\frac{a_3}{\alpha}\int_{\Omega}w^{\alpha+1}|\nabla\eta|^2dV.
		\end{align}
	\end{lemma}
	\begin{proof}
		By the regular theorem, away from \(\{w=0\}\), \(u\) is smooth and \eqref{eq:2.14} is valid. Let \(\varepsilon>0\) and set \(\psi=\frac{w_{\varepsilon}^{\alpha}\eta^2}{\left(\Phi'(w_{\varepsilon})\right)^2}\), where \(w_{\varepsilon}=(w-\varepsilon)^+\), \(\eta\in C_0^{\infty}(B_R(o))\) is a non-negative cutoff function, the  constant \(\alpha>2b_1\) will be determined later. Multiply  both sides of \eqref{eq:2.14} by  \(\psi\) and integrate over \(\Omega\), a direct computation yields
		\begin{equation}\label{eq:3.2}
			\begin{aligned}
				&\int_{\Omega}\left(\alpha-2w_{\varepsilon}h(w_{\varepsilon})\right) w_{\varepsilon}^{\alpha-1}\eta^2\left(|\nabla w|^2+2h(w)\langle\nabla u,\nabla w\rangle^2 \right)dV\\
				&+\int_{\Omega}2w_{\varepsilon}^{\alpha}\eta\langle\nabla w+2h(w)\langle\nabla u,\nabla w\rangle\nabla u,\nabla\eta\rangle dV+\frac{2}{n-1}\int_{\Omega}\left(\frac{G(w)}{\Phi'(w_{\varepsilon})}\right)^2w_{\varepsilon}^{\alpha}\eta^2dV\\
				\leq&2(n-1)\kappa\int_{\Omega}w_{\varepsilon}^{\alpha+1}\eta^2dV+2\delta\int_{\Omega}\frac{|\Phi'(w)||G(w)|}{\left(\Phi'(w_{\varepsilon})\right)^2}w_{\varepsilon}^{\alpha-\frac{1}{2}}\eta^2|\nabla w|dV.
			\end{aligned}
		\end{equation}
		Note that  the two terms containing inner products in the inequality \eqref{eq:3.2} can be controlled as follows
		\begin{align*}
			&\left(\alpha-2w_{\varepsilon}h(w_{\varepsilon})\right) w_{\varepsilon}^{\alpha-1}\eta^2\left(|\nabla w|^2+2h(w)\langle\nabla u,\nabla w\rangle^2 \right)\geq\min\left\{1,(1+2b_1)\left(1-\frac{2b_2}{\alpha}\right)\right\}\alpha w_{\varepsilon}^{\alpha-1}\eta^2|\nabla w|^2,
		\end{align*}
		and
		\begin{align*}
			2w_{\varepsilon}^{\alpha}\eta\langle\nabla w+2h(w)\langle\nabla u,\nabla w\rangle\nabla u,\nabla\eta\rangle&\geq-2(1+2b_2)w_{\varepsilon}^{\alpha}\eta|\nabla w||\nabla\eta|=-2a_2w_{\varepsilon}^{\alpha}\eta|\nabla w||\nabla\eta|.
		\end{align*}  
		Hence, letting \(\varepsilon\rightarrow0^+\)  in \eqref{eq:3.2} and using the above two inequalities, we obtain
		\begin{equation}\label{eq:3.3}
			\begin{aligned}
				&a_1\alpha\int_{\Omega}w^{\alpha-1}\eta^2|\nabla w|^2dV+\frac{2}{n-1}\int_{\Omega}\left(\frac{G(w)}{\Phi'(w)}\right)^2w^{\alpha}\eta^2dV\\
				\leq&2(n-1)\kappa\int_{\Omega}w^{\alpha+1}\eta^2dV+2a_2\int_{\Omega}w^{\alpha}\eta|\nabla w||\nabla\eta|dV\\
				&+2\delta\int_{\Omega}\frac{|\Phi'(w)||G(w)|}{\left(\Phi'(w)\right)^2}w_{\varepsilon}^{\alpha-\frac{1}{2}}\eta^2|\nabla w|dV.
			\end{aligned}
		\end{equation}   
		Let \(R_i\) denote the \(i\)-th term on the right-hand side of \eqref{eq:3.3}. By the Cauchy-Schwarz inequality and Young inequality,  
		\begin{align*}
			R_2&=2a_2\int_{\Omega}w^{\frac{1}{2}\alpha-\frac{1}{2}}\eta|\nabla w|\cdot w^{\frac{1}{2}\alpha+\frac{1}{2}}\left(\Phi'(w)\right)|\nabla\eta|dV\\
			&\leq2a_2\left(\int_{\Omega}w^{\alpha-1}\eta^2|\nabla w|^2dV\right)^{\frac{1}{2}}\left(\int_{\Omega}w^{\alpha+1}|\nabla\eta|^2dV\right)^{\frac{1}{2}}\\
			&\leq \frac{a_1\alpha}{4}\int_{\Omega}w^{\alpha-1}\eta^2|\nabla w|^2dV+\frac{4a_2^2}{a_1\alpha}\int_{\Omega}w^{\alpha+1}|\nabla\eta|^2dV,\\
			R_3&=2\int_{\Omega}\frac{\delta|\Phi'(w)||G(w)|}{\left(\Phi'(w)\right)^2}w^{\frac{\alpha}{2}}\eta\cdot w^{\frac{\alpha}{2}-\frac{1}{2}}\eta|\nabla w|dV\\
			&\leq2\left(\int_{\Omega}\left(\frac{\delta|\Phi'(w)||G(w)|}{\left(\Phi'(w)\right)^2}\right)^2w^{\alpha}\eta^2dV\right)^{\frac{1}{2}}\left(\int_{\Omega}w^{\alpha-1}\eta^2|\nabla w|^2dV\right)^{\frac{1}{2}}\\
			&\leq\frac{4\delta^2}{a_1\alpha}\int_{\Omega}\left(\frac{G(w)}{\Phi'(w)}\right)^2w^{\alpha}\eta^2dV+\frac{a_1\alpha}{4}\int_{\Omega}w^{\alpha-1}\eta^2|\nabla w|^2dV.
		\end{align*}  
		By choosing  \(\alpha\) such that  
		\(
		\frac{4\delta^2}{a_1\alpha}\leq\frac{1}{n-1},
		\)  
		we can derive from \eqref{eq:3.3} that
		\begin{equation}\label{eq:3.4}
			\begin{aligned}
				&\frac{a_1\alpha}{2}\int_{\Omega}w^{\alpha-1}\eta^2|\nabla w|^2dV+\frac{1}{n-1}\int_{\Omega}\left(\frac{G(w)}{\Phi'(w)}\right)^2w^{\alpha}\eta^2dV\\
				\leq&2(n-1)\kappa\int_{\Omega}w^{\alpha+1}\eta^2dV+\frac{4a_2^2}{a_1\alpha}\int_{\Omega}w^{\alpha+1}|\nabla\eta|^2dV.
			\end{aligned}
		\end{equation}  
		On the other hand, by using the inequality \((a+b)^2\leq 2(a^2+b^2)\), we have
		\begin{equation}\label{eq:3.5}
			\int_{\Omega}\left|\nabla\left(w^{\frac{\alpha}{2}+\frac{1}{2}}\eta\right)\right|^2dV\leq\frac{1}{2}(\alpha+1)^2\int_{\Omega}w^{\alpha-1}\eta^2|\nabla w|^2dV+2\int_{\Omega}w^{\alpha+1}|\nabla\eta|^2dV.
		\end{equation}  
		It follows immediately from \eqref{eq:3.4} and \eqref{eq:3.5} that
		\begin{equation}\label{eq:3.6}
			\begin{aligned}
				&\frac{a_1\alpha}{(\alpha+1)^2}\int_{\Omega}\left|\nabla\left(w^{\frac{\alpha}{2}+\frac{1}{2}}\eta\right)\right|^2dV+\frac{1}{n-1}\int_{\Omega}\left(\frac{G(w)}{\Phi'(w)}\right)^2w^{\alpha}\eta^2dV\\
				\leq&2(n-1)\kappa\int_{\Omega}w^{\alpha+1}\eta^2dV+\left(\frac{4a_2^2}{a_1\alpha}+\frac{2a_1\alpha}{(\alpha+1)^2}\right)\int_{\Omega}w^{\alpha+1}|\nabla\eta|^2dV.
			\end{aligned}
		\end{equation}  
		By choosing a suitable constant \(a_3\) depending only on \(\Phi,G,n\) such that 
		\[
		\frac{4a_2^2}{a_1\alpha}+\frac{2a_1\alpha}{(\alpha+1)^2}\leq\frac{a_3}{\alpha},
		\]  
		we complete the proof of Lemma \ref{lemma:integral inequality}.   
	\end{proof}
	\begin{lemma}\label{lemma:integral estimate}
		Let \((M^n,g)\) be a complete Riemannian manifold satisfying \({\rm Ric}_g\geq-(n-1)\kappa g\) for some constant \(\kappa\geq0\). Assume that \(u\) is solution to equation \eqref{eq:elliptic 1.1} and \(w=|\nabla u|^2\). If \(\Phi(w)\) and \(G(w)\) satisfy the same conditions Lemma \ref{lemma:G-Phi scaling}, then 		\begin{equation}\label{eq:3.7}
			e^{-\alpha_0}V^{\frac{2}{n}}\left(\int_{\Omega}w^{\frac{n}{n-2}(\alpha+1)}\eta^{\frac{2n}{n-2}}dV\right)^{\frac{n-2}{n}}+a_6\alpha R^2\int_{\Omega}w^{\alpha+2}\eta^2dV
			\leq a_7\left(\alpha_0^2\alpha\int_{\Omega}w^{\alpha+1}\eta^2dV+R^2\int_{\Omega}w^{\alpha+1}|\nabla\eta|^2dV\right),
		\end{equation}
		where
		\(
		\alpha_0=(1+\sqrt{\kappa}R)\max\left\{c_0+1,\frac{4(n-1)a_2^2}{a_1}\right\}.
		\)	
	\end{lemma}
	\begin{proof}
		By the Sobolev inequality \eqref{SCI}, we have
		\begin{equation}\label{eq:3.8}
			\left(\int_{\Omega}w^{\frac{n}{n-2}(\alpha+1)}\eta^{\frac{2n}{n-2}}dV\right)^{\frac{n-2}{n}}\leq e^{c_0(1+\sqrt{\kappa}R)}V^{-\frac{2}{n}}\left(R^2\int_{\Omega}\left|\nabla\left(w^{\frac{\alpha}{2}+\frac{1}{2}}\eta\right)\right|^2dV+\int_{\Omega}w^{\alpha+1}\eta^2dV\right),
		\end{equation}	
		where the constant \(c_0\) depends only on \(n\). Combining \eqref{eq:3.1} and the Sobolev inequality  \eqref{eq:3.8} yields
		\begin{equation}\label{eq:3.9}
			\begin{aligned}
				&\frac{a_1\alpha}{(\alpha+1)^2}e^{-c_0(1+\sqrt{\kappa}R)}V^{\frac{2}{n}}R^{-2}\left(\int_{\Omega}w^{\frac{n}{n-2}(\alpha+1)}\eta^{\frac{2n}{n-2}}dV\right)^{\frac{n-2}{n}}+\frac{1}{n-1}\int_{\Omega}\left(\frac{G(w)}{\Phi'(w)}\right)^2w^{\alpha}\eta^2dV\\
				\leq&2(n-1)\kappa\int_{\Omega}w^{\alpha+1}\eta^2dV+\frac{a_3}{\alpha}\int_{\Omega}w^{\alpha+1}|\nabla\eta|^2dV+\frac{a_1\alpha}{R^2(\alpha+1)^2}\int_{\Omega}w^{\alpha+1}\eta^2dV,	
			\end{aligned}
		\end{equation}	
		where we require that \(n\neq2\). Now we choose 
		\[
		c_1=\max\left\{c_0+1,\frac{4(n-1)a_2^2}{a_1}\right\}
		\]	
		and denote \(\alpha_0=c_1(1+\sqrt{\kappa}R)\). For \(\alpha\geq\alpha_0\), there exist constants \(a_4\) and \(a_5\) depending only on \(\Phi,n\) such that
		\[
		2(n-1)\kappa+\frac{a_1\alpha}{R^2(\alpha+1)^2}\leq\frac{a_4\alpha_0^2}{R^2},
		\quad
		and
		\quad
		\frac{a_5}{\alpha}\leq\frac{a_1\alpha}{(\alpha+1)^2}.
		\]	
		It follows that	
		\begin{align*}
			&\frac{a_5}{\alpha}e^{-\alpha_0}V^{\frac{2}{n}}R^{-2}\left(\int_{\Omega}w^{\frac{n}{n-2}(\alpha+1)}\eta^{\frac{2n}{n-2}}dV\right)^{\frac{n-2}{n}}+\frac{1}{n-1}\int_{\Omega}\left(\frac{G(w)}{\Phi'(w)}\right)^2w^{\alpha}\eta^2dV\\
			\leq&\frac{a_4\alpha_0^2}{R^2}\int_{\Omega}w^{\alpha+1}\eta^2dV+\frac{a_3}{\alpha}\int_{\Omega}w^{\alpha+1}|\nabla\eta|^2dV.
		\end{align*}	
		By applying Lemma \ref{lemma:G-Phi scaling} and taking  new suitable constants, we obtain the required result.	
	\end{proof}
	\begin{lemma}\label{lemma:L-estimate}
		Let \((M^n,g)\) be a complete Riemannian manifold satisfying \({\rm Ric}_g\geq-(n-1)\kappa g\) for some constant \(\kappa\geq0\).  Assume that \(u\) is solution to equation \eqref{eq:elliptic 1.1} and \(w=|\nabla u|^2\). If \(\Phi(w)\) and \(G(w)\) satisfy the same conditions in Lemma \ref{lemma:G-Phi scaling}, then 
		\begin{equation}\label{eq:3.10}
			\begin{aligned}
				\|w\|_{L^{\alpha_1}(B_{\frac{3R}{4}}(o))}\leq a_{12}V^{\frac{1}{\alpha_1}}\left(\frac{1+\sqrt{\kappa}R}{R}\right)^{\alpha_0+2}
			\end{aligned},
		\end{equation}
		where \(\alpha_1=\frac{n}{n-2}(\alpha_0+1)\) and the constant \(a_{12}\) depends only on \(n,\Phi,G\).
	\end{lemma}
	\begin{proof}
		Set \(\alpha=\alpha_0\) in \eqref{eq:3.7} to obtain
		\begin{equation}\label{eq:3.11}
			e^{-\alpha_0}V^{\frac{2}{n}}\left(\int_{\Omega}w^{\frac{n}{n-2}(\alpha+1)}\eta^{\frac{2n}{n-2}}dV\right)^{\frac{n-2}{n}}+a_6\alpha_0 R^2\int_{\Omega}w^{\alpha+2}\eta^2dV\\
			\leq a_7\left(\alpha_0^3\int_{\Omega}w^{\alpha+1}\eta^2dV+R^2\int_{\Omega}w^{\alpha+1}|\nabla\eta|^2dV\right).
		\end{equation}
		If
		\(
		w\geq\frac{2a_7\alpha_0^2}{a_6R^2},
		\)
		it is easy to see that
		\[
		a_7\alpha_0^3w^{\alpha_0+1}\eta^2\leq\frac{1}{2}a_6\alpha_0R^2w^{\alpha_0+2}\eta^2.
		\]
		Let \(\Omega=\Omega_1\cup\Omega_2\), where
		\[
		\Omega_1=\left\{x\in\Omega:w\geq\frac{2a_7\alpha_0^2}{a_6R^2}\right\},
		\]
		and \(\Omega_2\) is the complement of \(\Omega_1\). Then, 
		\begin{equation}\label{eq:3.12}
			\begin{aligned}
				a_7\alpha_0^3\int_{\Omega}w^{\alpha_0+1}\eta^2dV&=a_7\alpha_0^3\int_{\Omega_1}w^{\alpha_0+1}\eta^2dV+a_7\alpha_0^3\int_{\Omega_2}w^{\alpha_0+1}\eta^2dV\\
				&\leq\frac{a_6\alpha_0R^2}{2}\int_{\Omega}w^{\alpha_0+2}\eta^2dV+\frac{a_6\alpha_0R^2}{2}\left(\frac{2a_7\alpha_0^2}{a_6R^2}\right)^{\alpha_0+2}V,	
			\end{aligned}
		\end{equation}
		where \(V\) is the volume of \(\Omega\).
		
		Choose \(\tilde{\eta}\) such that \(0\leq\tilde{\eta}\leq1\), \(\tilde{\eta}\in C_0^{\infty}(B_R(o))\), \(\tilde{\eta}\equiv1\) in \(B_{\frac{3R}{4}}(o)\) and \(|\nabla\tilde{\eta}|\leq\frac{C}{R}\). Let
		\(
		\eta=\tilde{\eta}^{\alpha_0+2}.
		\)
		A direct calculation  gives
		\begin{equation}\label{eq:3.13}
			a_7R^2|\nabla\eta|^2\leq a_7C^2(\alpha_0+2)^2\eta^{\frac{2(\alpha_0+1)}{\alpha_0+2}}\leq a_8\alpha_0^2\eta^{\frac{2(\alpha_0+1)}{\alpha_0+2}}.
		\end{equation}
		Then using H\"older inequality and Young equality, \eqref{eq:3.13} implies
		\begin{equation}\label{eq:3.14}
			\begin{aligned}
				a_7R^2\int_{\Omega}w^{\alpha_0+1}|\nabla\eta|^2dV&\leq a_8\alpha_0^2\int_{\Omega}w^{\alpha_0+1}\eta^{\frac{2(\alpha_0+1)}{\alpha_0+2}}dV\\
				&\leq a_8\alpha_0^2\left(\int_{\Omega}w^{\alpha_0+2}\eta^2dV\right)^{\frac{\alpha_0+1}{\alpha_0+2}}V^{\frac{1}{\alpha_0+2}}\\
				&\leq\frac{a_6\alpha_0R^2}{2}\int_{\Omega}w^{\alpha_0+2}\eta^2dV+\frac{a_6\alpha_0R^2}{2}\left(\frac{2a_8\alpha_0}{a_6R^2}\right)^{\alpha_0+2}V.
			\end{aligned}
		\end{equation}
		Combining  \eqref{eq:3.12} and \eqref{eq:3.14} with \eqref{eq:3.7}, we get 
		\begin{equation}\label{eq:3.15}
			\begin{aligned}
				\left(\int_{\Omega}w^{\frac{n}{n-2}(\alpha_0+1)}\eta^{\frac{2n}{n-2}}dV\right)^{\frac{n-2}{n}}
				\leq&e^{\alpha_0}V^{1-\frac{2}{n}}\left[\frac{a_6\alpha_0R^2}{2}\left(\frac{2a_7\alpha_0^2}{a_6R^2}\right)^{\alpha_0+2}+\frac{a_6\alpha_0R^2}{2}\left(\frac{2a_8\alpha_0}{a_6R^2}\right)^{\alpha_0+2}\right]\\
				\leq&a_9e^{\alpha_0}V^{1-\frac{2}{n}}a_{10}^{\alpha_0}\alpha_0^3\left(\frac{\alpha_0}{R}\right)^{2(\alpha_0+1)}.
			\end{aligned}
		\end{equation} 
		Taking the\(1/(\alpha_0+1)\)-th power on both sides of \eqref{eq:3.15} yields 
		\begin{equation}\label{eq:3.16}
			\|w\|_{L^{\alpha_1}(B_{\frac{3R}{4}}(o))}\leq a_{11}V^{\frac{1}{\alpha_1}}\left(\frac{\alpha_0}{R}\right)^{2}.
		\end{equation}
		Hence, the require inequality.	
	\end{proof}
	
	\begin{proof}[\bf Proof of Theorem \ref{theorem:elliptic equation key}]
		Discarding the second term on the left-hand side of \eqref{eq:3.7} gives
		\begin{equation}\label{eq:3.17}
			\begin{aligned}
				\left(\int_{\Omega}w^{\frac{n}{n-2}(\alpha+1)}\eta^{\frac{2n}{n-2}}dV\right)^{\frac{n-2}{n}}\leq e^{\alpha_0}V^{-\frac{2}{n}}a_7\left(\alpha_0^2\alpha\int_{\Omega}w^{\alpha+1}\eta^2dV+R^2\int_{\Omega}w^{\alpha+1}|\nabla\eta|^2dV\right).
			\end{aligned}
		\end{equation}	
		To apply the Nash-Moser iteration, set	
		\[
		\alpha_{l+1}=\frac{n}{n-2}\alpha_l,\quad\Omega_l=B(o,r_l),\quad r_l=\frac{R}{2}+\frac{R}{4^l},\quad l=1,2,...,
		\]	
		and choose \(\eta_l\in C_0^{\infty}(\Omega_l)\) such that
		\[
		\eta_l\equiv1\quad \text{in}\quad\Omega_{l+1},\quad0\leq\tilde{\eta}_l\leq1,\quad|\nabla\eta_l|\leq\frac{C(n)4^l}{R}.
		\]	
		Now choose \(\alpha\) such that \(\alpha+1=\alpha_l\) and take \(\eta=\eta_l\) in \eqref{eq:3.17}. Using the fact
		\[
		\alpha_l\alpha_0^2\eta_l^2+R^2|\nabla\eta_l|^2\leq\alpha_0^2(\alpha_0+1)\left(\frac{n}{n-2}\right)^l+C^2(n)16^l\leq a_{12}^l\alpha_0^3\alpha_1
		\]	
		we have
		\begin{align*}
			\left(\int_{\Omega_{l+1}}w^{\alpha_{l+1}}dV\right)^{\frac{1}{\alpha_{l+1}}}&\leq\left(a_7e^{\alpha_0}V^{-\frac{2}{n}}\right)^{\frac{1}{\alpha_l}}\left(\int_{\Omega_l}(\alpha_l\alpha_0^2\eta_l^2+R^2|\nabla\eta_l|^2)w^{\alpha_l}dV\right)^{\frac{1}{\alpha_l}}\\
			&\leq\left(a_7\alpha_0^2\alpha_1e^{\alpha_0}V^{-\frac{2}{n}}\right)^{\frac{1}{\alpha_l}}a_{12}^{\frac{1}{\alpha_l}}\left(\int_{\Omega_l}w^{\alpha_l}dV\right)^{\frac{1}{\alpha_l}}.
		\end{align*}	
		By the facts
		\[
		\sum_{l=1}^{\infty}\frac{1}{\alpha_l}=\frac{n}{2\alpha_1}=\frac{n-2}{2(\alpha_0+1)},\quad\sum_{l=1}^{\infty}\frac{l}{\alpha_l}=\frac{n^2}{4\alpha_1}\frac{n(n-2)}{4(\alpha_0+1)}\leq\frac{n(n-2)}{4},
		\]	
		the  quantities
		\[
		\left(a_7\alpha_0^2\alpha_1e^{\alpha_0}\right)^{\sum_{l=1}^{\infty}\frac{1}{\alpha_l}}\quad\text{and}\quad a_{12}^{\sum_{l=1}^{\infty}\frac{1}{\alpha_l}},
		\]	
		are both uniformly bounded for any \(R>0\) and \(\kappa\geq0\). By a standard iteration procedure, we obtain
		\begin{equation}\label{eq:3.18}
			\|w\|_{L^{\infty}(B_{\frac{R}{2}}(o))}\leq\left(a_7\alpha_0^2\alpha_1e^{\alpha_0}V^{-\frac{2}{n}}\right)^{\sum_{l=1}^{\infty}\frac{1}{\alpha_l}}a_{12}^{\sum_{l=1}^{\infty}\frac{1}{\alpha_l}}\|w\|_{L^{\alpha_1}(B_{\frac{3R}{4}}(o))}
			\leq a_{13}V^{-\frac{1}{\alpha_1}}\|w\|_{L^{\alpha_1}(B_{\frac{3R}{4}}(o))}.
		\end{equation}	
		Combining \eqref{eq:3.18} with \eqref{eq:3.16} leads to
		\begin{equation}\label{eq:3.19}
			\|\nabla u\|_{L^{\infty}(B_{\frac{R}{2}}(o))}\leq a_{14}\left(\frac{1+\sqrt{\kappa}R}{R}\right).
		\end{equation}	
		This completes the proof of Theorem \ref{theorem:elliptic equation key}.		
	\end{proof}
	\begin{corollary}
		Let \(\Omega\subset\mathbb{R}^n\) be a domain and \(u\) be a solution to the equation \eqref{eq:elliptic 1.1} with \(n\geq2\), defined on a geodesic ball \(B_R(o)\). Under the same assumptions as Theorem \ref{theorem:elliptic equation key},  we have
		\[
		|\nabla u(x)|\leq c(n,\Phi,G)(d(x,\partial\Omega))^{-1}
		\]
		for any \(x\in\Omega\). In addition, if \(\Omega=\mathbb{R}^n\), then \(u\) is a constant.
	\end{corollary}
	\begin{proof}
		Denote \(R=d(x,\partial\Omega)\). Obviously, we have \(B_R(x)\subset\Omega\) and 
		\[
		|\nabla u(x)|\leq \sup_{B_{\frac{R}{2}(o)}}|\nabla u|\leq a_{16}\left(\frac{1}{R}\right)=a_{16}(d(x,\partial\Omega))^{-1}
		\]
		Here the constant \(a_{16}\) depends only on \(n,\Phi,G\).	
	\end{proof}
	\begin{corollary}
		Assume \((M^n,g)\) satisfies the same assumptions as in Theorem \ref{theorem:elliptic equation key}. Let \(u\) be a global solution to the equation \eqref{eq:elliptic 1.1} on \(M\). Then, for fixed \(o\in M\) and any \(x\in M\) we have 
		\begin{equation}\label{eq:3.20}
			u(o)-c(n,\Phi,G)\kappa^{\frac{1}{2}}d(x,o)\leq u(x)\leq u(o)+c(n,\Phi,G)\kappa^{\frac{1}{2}}d(x,o).	
		\end{equation}
	\end{corollary}
	\begin{proof}
		Letting \(R\rightarrow\infty\) in \eqref{eq:3.19} yields
		\begin{equation}\label{eq:3.21}
			|\nabla u|\leq c(n,\Phi,G)\kappa^{\frac{1}{2}}.	
		\end{equation}
		Let \(d=d(x,o)\) be the distance between \(o\) and \(x\). For any length-minimizing geodesic segment \(\gamma(t):[0,d]\rightarrow M\) connecting \(o\) and \(x\), we have 
		\begin{equation}\label{eq:3.22}
			u(x)=u(o)+\int_{0}^{d}\frac{d}{dt}(u\circ\gamma(t))dt.
		\end{equation}
		From \eqref{eq:3.21}, 
		\begin{equation}\label{eq:3.23}
			\left|\frac{d}{dt}(u\circ\gamma(t))\right|\leq|\nabla u(\gamma(t))||\gamma'(t)|\leq c(n,\Phi,G)\kappa^{\frac{1}{2}}	.
		\end{equation}
		It follows from \eqref{eq:3.22} and \eqref{eq:3.23} that
		\[
		|u(x)-u(o)|\leq c(n,\Phi,G)\kappa^{\frac{1}{2}}d(x,o)
		\]
		which implies \eqref{eq:3.20}.
	\end{proof}

	\section{Gradient estimates for the nonlinear parabolic equation}\label{S4}
	
	In this section, we consider the general nonlinear parabolic equation
	\begin{equation}\label{eq:4.1}
		\frac{\partial u}{\partial t}=\Lambda_{\Phi}u:=\Delta_{\Phi}u-\frac{p}{2}\Phi(|\nabla u|^2).
	\end{equation}
	We say that a function \(u=u(x,t)\) is a \emph{weak solution} of \eqref{eq:4.1}
	on \(\Gamma\times I\) if
	\(u\in L_{\mathrm{loc}}^{2}(\Gamma\times I)\),
	\(\nabla u\in L_{\mathrm{loc}}^{2}(\Gamma\times I)\),
	\(\Phi'(|\nabla u|^2)|\nabla u|\in L_{\mathrm{loc}}^{1}(\Gamma\times I)\),
	and
	\[
	\int_{I}\int_{\Gamma} u\,\partial_{t}\psi\,dxdt
	= -\int_{I}\int_{\Gamma}\Phi'(|\nabla u|^2)\langle\nabla u,\nabla\psi\rangle\,dxdt
	- \int_{I}\int_{\Gamma}\frac{p}{2}\Phi(|\nabla u|^2)\psi\,dxdt
	\]
	holds for all \(\psi\in C_{0}^{\infty}(\Gamma\times I)\).
	
	It is worth mentioning that, under suitable ellipticity and growth conditions on
	\(\Phi\) (see \cite{Miao2026}), any such weak solution satisfies
	\[
	u\in L_{\mathrm{loc}}^{2}\bigl(I;W_{\mathrm{loc}}^{2,2}(\Gamma)\bigr)
	\cap C_{\mathrm{loc}}^{1,\alpha}(\Gamma\times I)
	\]
	for some \(\alpha\in(0,1)\), and is in fact smooth away from the set
	\(\{\nabla u=0\}\).  Further discussion on the regularity of solutions can be found
	in \cite{Miao2026}.

	Set \(w = |\nabla u|^{2}\), the linearized operator \(\mathscr{L}\) of $\Lambda_{\Phi}$ at point $u$ is defined as
	\begin{equation}\label{eq:4.2}	
		\mathscr{L}(\psi):={\rm div}(\Phi'(w)A(\nabla\psi)) - p\Phi'(w)\langle{\nabla u},{\nabla \psi}\rangle,
	\end{equation}	
	and the corresponding  parabolic operator is given by
	\[
	\Box_{\Phi}:=\frac{\partial}{\partial t}-\mathscr{L},
	\]
	where $A$ is the tensor defined in \eqref{eq:2.1}.
	
	\begin{lemma}
		If \(u\) is a solution of \eqref{eq:4.1} and \(w=|\nabla u|^2\), then 
		\begin{equation}\label{eq:4.3}
			\mathscr{L}(w)=2\Phi'(w)\left(|\nabla\nabla u|^2+{\rm Ric}(\nabla u,\nabla u)\right)+\Phi''(w)|\nabla w|^2+2\langle\nabla u,\nabla\Lambda_{\Phi}(u)\rangle,
		\end{equation}
		and
		\begin{equation}\label{eq:4.5}
			\mathscr{L}(\Phi(w))=2\left(\Phi'(w)\right)^2\left(|\nabla\nabla u|_A^2+{\rm Ric}(\nabla u,\nabla u)\right)+2\Phi'(w)\langle\nabla\Lambda_{\Phi}u,\nabla u\rangle,
		\end{equation}

	\end{lemma}
	\begin{proof}
		Using the \(\Phi\)-Bochner formula \eqref{eq:2.5} and definition of \(\Lambda_{\Phi}\) in \eqref{eq:4.1}, we have
		\begin{align*}
			\mathscr{L}(w)&={\rm div}\left(\Phi'(w)A(\nabla w)\right)-p\Phi'(w)\langle\nabla u,\nabla w\rangle\\
			&=2\Phi'(w)\left(|\nabla\nabla u|^2+{\rm Ric}(\nabla u,\nabla u)\right)+2\langle\nabla\Delta_{\Phi}u,\nabla u\rangle+\Phi''(w)|\nabla w|^2-p\Phi'(w)\langle\nabla u,\nabla w\rangle\\
			&=2\Phi'(w)\left(|\nabla\nabla u|^2+{\rm Ric}(\nabla u,\nabla u)\right)+\Phi''(w)|\nabla w|^2+2\langle\nabla u,\nabla\Lambda_{\Phi}(u)\rangle.
		\end{align*}
		Formula \eqref{eq:4.5} follows from a direct calculation using \eqref{eq:4.3}.
	\end{proof}
	\begin{lemma}
		If \(u\) is a solution of \eqref{eq:4.1} and \(w=|\nabla u|^2\), then 
		\begin{equation}\label{eq:4.6}
			\Box_{\Phi}(u_t)=0,
		\end{equation}
	and
	\begin{equation}\label{eq:4.7}
		\Box_{\Phi}\Phi(w)=-2\left(\Phi'(w)\right)^2\left(|\nabla\nabla u|_A^2+{\rm Ric}(\nabla u,\nabla u)\right).
	\end{equation}
\end{lemma}
\begin{proof}
	A direct calculation gives
	\begin{align*}
		\frac{\partial}{\partial t}\left(\Delta_{\Phi}u\right)&={\rm div}\left((\Phi'(w)\nabla u_t+\Phi''(w)w_t\nabla u)\right)\\
		&={\rm div}\left(\Phi'(w)(\nabla u_t+2h(w)\langle\nabla u,\nabla u_t\rangle\nabla u)\right)\\
		&={\rm div}\left(\Phi'(w)A(\nabla u_t)\right),
	\end{align*}
	and
	\begin{align*}
		\frac{\partial}{\partial t}\left(\Phi(w)\right)=\Phi'(w)\frac{\partial w}{\partial t}=2\Phi'(w)\langle\nabla u,\nabla u_t\rangle.
	\end{align*}
	Combining two identities with the definitions of \(\mathscr{L}\) and \(\Lambda_{\Phi}\) yields
	\[
	\Box_{\Phi}u_t=\partial_tu_t-\mathscr{L}(u_t)=\partial_t(u_t-\Lambda_{\Phi}u)=0.
	\]
	Formula  \eqref{eq:4.7} follows directly from  \eqref{eq:4.6} and \eqref{eq:4.5}.
\end{proof}
\begin{proposition}\label{corollary:H_a}
	Define the Harnack quantity
	\[
	H:=\frac{p}{2}\Phi(w)+ u_t=\Delta_{\Phi}u.
	\]
	Then,
	\begin{equation}\label{eq:4.8}
		\Box_{\Phi}H=-p\left(\Phi'(w)\right)^2\left(|\nabla\nabla u|_A^2+{\rm Ric}(\nabla u,\nabla u)\right).
	\end{equation}
\end{proposition}

\begin{proof}[\bf Proof of Theorem \ref{corollary:parabolic 1.1}]
	Assume that \(tH\) attains its maximum at point \((x_0,t_0)\). Using \eqref{eq:4.8}, we have
	\begin{align*}
		0&\leq\Box_{\Phi}(tH)\\
		&=H-pt\left(\Phi'(w)\right)^2\left(|\nabla\nabla u|_A^2+{\rm Ric}(\nabla u,\nabla u)\right)\\
		&\leq H-\frac{pt}{n}H^2,
	\end{align*}
	where we used the non-negativity of Ricci curvature and Cauchy-Schwarz inequality
	\begin{equation}\label{eq:trA scaling}
		\left(\Phi'(w)\right)^2|\nabla\nabla u|_A^2\geq\frac{1}{n}\left({\rm tr}_A\Phi'(w)\nabla\nabla u\right)^2=\frac{1}{n}\left(\Delta_{\Phi}u\right)^2.
	\end{equation}
	where \({\rm tr}_A\Phi'(w)\nabla\nabla u=\sum_{i,j}A^{ij}\Phi'(w)\nabla_i\nabla_ju\). Thus, we obtain the estimate \eqref{pe}.
\end{proof}

  \section{The case \({\rm dim}(M)=n=2\)}

In the proof of Theorem \ref{theorem:elliptic equation key} we used the Saloff-Coste Sobolev inequality 
with exponent $\frac{2n}{n-2}$, which requires $n>2$. 
Here we sketch the necessary modifications when $\dim M = 2$.

For $n=2$ we apply Lemma \ref{lemma:Saloff-Coste} with an auxiliary dimension $n'=4$, which yields
\begin{equation}\label{eq:sob2}
	\Bigl(\int_{B}|f|^{4}\,dV\Bigr)^{1/2}
	\le e^{c_0(1+\sqrt{\kappa}R)}\,V^{-\frac12}R^{2}
	\Bigl(\int_{B}|\nabla f|^{2}+R^{-2}f^{2}\,dV\Bigr)
\end{equation}
for every $f\in C_{0}^{\infty}(B)$.  
Replacing $f$ by $w^{\frac{\alpha+1}{2}}\eta$ and inserting this into the
basic integral inequality of Lemma~3.1 we obtain, for $\alpha\ge\alpha_{0}$,
\begin{align}
	&e^{-\alpha_{0}}V^{\frac12}\Bigl(\int_{\Omega}w^{2(\alpha+1)}\eta^{4}\,dV\Bigr)^{\frac12}
	+ a_{6}\alpha R^{2}\int_{\Omega}w^{\alpha+2}\eta^{2}\,dV \notag\\
	&\qquad \le a_{7}\Bigl(R^{2}\int_{\Omega}w^{\alpha+1}|\nabla\eta|^{2}\,dV
	+ \alpha_{0}^{2}\alpha\int_{\Omega}w^{\alpha+1}\eta^{2}\,dV\Bigr),
	\label{eq:n2ell}
\end{align}
where $\alpha_{0}=c_{1}(1+\sqrt{\kappa}R)$ as before (with a possibly larger constant $c_{1}$).  

\medskip
\noindent\textit{$L^{\alpha_{1}}$-estimate.}  
Set $\alpha=\alpha_{0}$ in \eqref{eq:n2ell} and choose $\eta$ to be a
cut‑off function supported in $B_{R}$ with $\eta\equiv1$ on $B_{3R/4}$.
Exactly the same splitting argument as in Lemma~3.2 gives
\begin{equation}\label{eq:n2Lalpha}
	\|w\|_{L^{\alpha_{1}}(B_{3R/4})}
	\le a_{12}\,V^{\frac1{\alpha_{1}}}\Bigl(\frac{1+\sqrt{\kappa}R}{R}\Bigr)^{2},
	\qquad \alpha_{1}=2(\alpha_{0}+1).
\end{equation}

\medskip
\noindent\textit{Nash--Moser iteration.}  
We set $\alpha_{k+1}=2\alpha_{k}$ ($k\ge1$), choose concentric balls
$B_{r_{k}}$ with $r_{k}=R/2+R/4^{k}$ and cut‑off functions $\eta_{k}$
adapted to them.  From \eqref{eq:n2ell} we obtain, for every $k$,
\[
\Bigl(\int_{B_{r_{k+1}}}w^{\alpha_{k+1}}\,dV\Bigr)^{\frac1{\alpha_{k+1}}}
\le \bigl(a_{7}\alpha_{0}^{2}\alpha_{1}e^{\alpha_{0}}V^{-\frac12}\bigr)^{\frac1{\alpha_{k}}}
a_{13}^{\frac{k}{\alpha_{k}}}
\Bigl(\int_{B_{r_{k}}}w^{\alpha_{k}}\,dV\Bigr)^{\frac1{\alpha_{k}}}.
\]
Because $\sum\frac1{\alpha_{k}}=\frac1{\alpha_{1}}\sum 2^{-k}= \frac2{\alpha_{1}}$
and $\sum\frac{k}{\alpha_{k}}<\infty$, the constants remain bounded as $k\to\infty$.
Iterating yields
\begin{equation}\label{eq:n2iter}
	\|w\|_{L^{\infty}(B_{R/2})}
	\le a_{14}\,V^{-\frac1{\alpha_{1}}}\|w\|_{L^{\alpha_{1}}(B_{3R/4})}.
\end{equation}
Combining \eqref{eq:n2Lalpha} and \eqref{eq:n2iter} we finally obtain
\[
\sup_{B_{R/2}}|\nabla u|
\le C_{n,\Phi,G}\Bigl(\frac{1+\sqrt{\kappa}R}{R}\Bigr).
\]
Thus Theorem \ref{theorem:elliptic equation key} holds for $n=2$ as well.  The Harnack inequality and
Liouville theorem follow in exactly the same way as for $n>2$.

Yu-Zhao Wang,\\
School of Mathematics and  Statistics,  Shanxi University, Taiyuan, 030006, Shanxi, China

\emph{E-mail:} wangyuzhao@sxu.edu.cn

Jian-Hua Hao , \\
School of Mathematics and  Statistics, Shanxi University, Taiyuan, 030006, Shanxi, China.

\emph{E-mail:haojianhua@sxu.edu.cn}

\end{document}